\documentclass[12pt]{amsart}
\usepackage[all]{xy}
\usepackage{a4wide}
\usepackage{mathtools}
\usepackage{amssymb}
\usepackage{amsthm}
\usepackage{amsmath}
\usepackage{mathrsfs}
\usepackage{mathabx}
\usepackage{theoremref}
\usepackage{scalerel}
\usepackage{amscd}
\usepackage{verbatim}
\usepackage[all]{xy}
\usepackage{hyperref}

\usepackage{comment}

\usepackage[textsize=tiny]{todonotes}

\usepackage{enumitem}

\numberwithin{equation}{section}

\DeclareMathOperator{\GL}{GL}

\DeclareMathOperator{\G}{\Gamma}

\DeclareMathOperator{\SL}{SL}
\DeclareMathOperator{\Mp}{Mp}

\DeclareMathOperator{\N}{\mathbb{N}}
\DeclareMathOperator{\Z}{\mathbb{Z}}
\DeclareMathOperator{\R}{\mathbb{R}}

\renewcommand{\H}{\mathbb{H}}

\DeclareMathOperator{\mero}{mero}

\DeclareMathOperator{\sgn}{sgn}

\newcommand{\reg}{\mathrm{reg}}

\newcommand{\bs}{\ensuremath{\backslash}}

\makeatletter
\newcommand*\defbb[1]{
  \expandafter\newcommand\csname I#1\endcsname{\mathbb{#1}}}
\newcommand*\defbbs[1]{
  \@for\@i:=#1\do{\expandafter\defbb\expandafter{\@i}}}
\makeatother
\defbbs{A,B,C,D,E,F,G,H,I,K,L,M,N,O,P,Q,R,S,T,U,V,W,X,Y,Z} 

\makeatletter
\newcommand*\deffrak[1]{
  \expandafter\newcommand\csname frak#1\endcsname{\mathfrak{#1}}}
\newcommand*\deffraks[1]{
  \@for\@i:=#1\do{\expandafter\deffrak\expandafter{\@i}}}
\makeatother
\deffraks{a,b,c,d,e,f,g,h,i,j,k,l,m,n,o,p,q,r,s,t,u,v,w,x,y,z,A,B,C,D,E,F,G,H,I,J,K,L,M,N,O,P,Q,R,S,T,U,V,W,X,Y,Z}

\makeatletter
\newcommand*\defcal[1]{
  \expandafter\newcommand\csname cal#1\endcsname{\mathcal{#1}}}
\newcommand*\defcals[1]{
  \@for\@i:=#1\do{\expandafter\defcal\expandafter{\@i}}}
\makeatother
\defcals{A,B,C,D,E,F,G,H,I,J,K,L,M,N,O,P,Q,R,S,T,U,V,W,X,Y,Z}

  \newtheorem{Theorem}{Theorem}[section]
  \newtheorem{Lemma}[Theorem]{Lemma}
  \newtheorem{Proposition}[Theorem]{Proposition} 
  \newtheorem{Corollary}[Theorem]{Corollary}
  
  \theoremstyle{definition}

  \newtheorem{Remark}[Theorem]{Remark}

\date{\today}

\title{Magnetic orthogonal modular forms}

\author{Claudia Alfes}

\address{Universität Bielefeld, Fakultät für Mathematik, Postfach 100 131, 33501 Bielefeld, Germany}
\email{alfes@math.uni-bielefeld.de}

\author{Paul Kiefer}

\address{Department of Mathematics, University of Antwerp, BE-2000 Antwerp, Belgium.}
\email{Paul.Kiefer@uantwerpen.be}

\thanks{The research of the authors is funded by the Deutsche Forschungsgemeinschaft (DFG, German Research Foundation) -- SFB-TRR 358/1 2023 -- 491392403. The second author was funded by the Research Foundation - Flanders (FWO) within the framework of the Odysseus program project number G0D9323N}

\begin{document}

\begin{abstract}
    In this note we show that certain meromorphic orthogonal modular forms are magnetic, i.e.\ their Fourier coefficients satisfy special divisibility criteria. These meromorphic orthogonal modular forms are counterparts to the orthogonal cusp forms considered by Oda. We show that the seminal work of Borcherds implies the magneticity of these forms.

\medskip

\noindent\textsc{R\'esum\'e.}
Dans cette note, nous montrons que certaines formes modulaires orthogonales méromorphes sont magnétiques, c’est-à-dire que leurs coefficients de Fourier satisfont à des critères particuliers de divisibilité. Ces formes modulaires orthogonales méromorphes sont les analogues des formes cuspidales orthogonales considérées par Oda. Nous montrons que les travaux fondateurs de Borcherds impliquent la magnéticité de ces formes.

\end{abstract}
\maketitle


\section{Introduction}
The coefficients of the modular $j$-invariant $a_j(n)$ satisfy the following divisibility property
\[
2^{3a+8}3^{2b+5}5^{c+1}7^d\,  |\, a_j(2^a3^b5^c7^d n)
\]
for all $a,b,c,d\in\IN_0$ and $n\in\IN$. This divisibility phenomenon also holds for higher weight weakly holomorphic modular forms, that is, modular forms with poles of finite order at the cusps. Duke and Jenkins \cite{dukejenkins} proved that the Fourier coefficients $a(n)$ of certain weakly holomorphic modular forms of positive weight $2k$ satisfy the divisibility property $n^{k-1}|a(n)$ for all $n\in\IN$. 

In \cite{broadhurstzudilin} Broadhurst and Zudilin observed that the $n$-th Fourier coefficient of a certain meromorphic modular form of weight $4$ on the congruence subgroup $\G_0(8)$ that arose from a physics problem in \cite{ausserlechner} is divisible by $n$. Meromorphic modular forms are modular forms that are allowed to have finite order poles on the upper half-plane $\IH$. Inspired by the physics background, meromorphic modular forms of positive weight $2k$ whose Fourier coefficients $c(n)$ are divisible by $n^{k-1}$ for all $n$ are referred to as \emph{magnetic modular forms}\footnote{Here, we stick to the definition of magneticity used in \cite{loebrichschwagenscheidt}. Note that there is a more general notion of being \textit{magnetic of depth $d$} which is for example introduced in \cite{magneticconjectures}.}.

The magneticity of the form of Broadhurst and Zudilin was proven by Li and Neururer in \cite{lineururer} using the theory of Borcherds' regularized additive theta lift of weakly holomorphic modular forms and Hecke theory. Their method gave rise to many other examples of magnetic modular forms, including the weight $4$ form $E_4\Delta/E_6^2$ and the weight $6$ form $E_6\Delta/E_4^3$ (see \cite{pasolzudilin}). Here, $E_4$ and $E_6$ denote the classical Eisenstein series of weight $4$ and $6$ respectively, and $\Delta(z) =(E_4^3-E_6^2)/1728$ is Ramanujan's weight $12$ $\Delta$-function.

In \cite{loebrichschwagenscheidt} L\"obrich and Schwagenscheidt showed that specific linear combinations of certain meromorphic modular forms associated with binary quadratic forms are magnetic. 
For an integer $d\equiv 0,1\pmod{4}$ and a fundamental discriminant $D\in\IZ$ satisfying $dD<0$ let $\mathcal{Q}_{dD}$ be the set of positive definite integral binary quadratic forms of discriminant $dD$. 
Furthermore, let $\chi_D$ be the usual genus character in $\mathcal{Q}_{dD}$ as defined in \cite{gkz}. 
For $k\in \IN$ with $k\geq 2$, L\"obrich and Schwagenscheidt considered the functions
\[
    f_{k,d,D}(z)= C_{k,d,D} \sum_{Q\in\mathcal{Q}_{dD}} \chi_D(Q) Q(z,1)^{-k},
\]
where $C_{k,d,D}$ is an explicitly given (normalisation) constant.
These were first defined by Bengoechea in \cite{Paloma}, whereas the holomorphic counterparts of these functions (corresponding to positive discriminants $dD$) were first considered by Zagier \cite{zagierdoinaganuma} as elliptic specialisations of certain Hilbert modular forms.

Even more general, the functions $f_{k,d,D}(z)$ are elliptic specialisations of a family of meromorphic orthogonal modular forms recently defined by Kiefer \cite{kiefer}. In this note, we show that these meromorphic orthogonal forms are magnetic in signature $(2,n)$ with $n$ even.

Our results follow from the theory of Borcherds' liftings \cite{borcherds}. In Theorem 14.3 of loc.\ cit.\ Borcherds gives the Fourier expansion of an orthogonal modular form that is given as a certain theta lifting of vector-valued weakly holomorphic modular forms of weight $k$. Assuming that the Fourier coefficients of the input weakly holomorphic modular forms satisfy certain divisibility properties immediately gives the magneticity of their theta liftings (see Corollary \ref{corollary:divborcherds}).

We now describe our results in more detail. 
We let $L$ be an even lattice of signature $(2, n)$, $n\in\IN$ even. We denote the corresponding quadratic form by $q$ and the corresponding bilinear form by $(\cdot,\cdot)$. 
By $L'$ we denote the dual lattice of $L$.
We choose a primitive isotropic vector $e\in L$ and let $e'\in L'$ with $(e,e')=1$. 
We can realize the hermitian symmetric domain corresponding to the orthogonal group $O(L\otimes \IR)$ by
\[
 \IH_n^\pm =\left\{Z=X+iY \in K\otimes \IC \,|\, q(Y)>0\right\},
\]
where $K=e^\perp\cap e'^\perp \cap L$. 
There is a natural action of the orthogonal group $O(L\otimes \IR)$ on $\H_n^\pm$. 
By $\IH_n$ we denote a connected component of $\IH_n^\pm$.
We write $O^+(L\times \IR)$ for the subgroup of $O(L\otimes\IR)$ of index two that preserves the component. 
Moreover, let $\G\subset O^+(L)=O(L)\cap O^+(l\otimes \IR)$ be the discriminant kernel, i.e. the subgroup that acts trivially on the cosets $L' / L$.
Let $m<0$ and $\beta \in L' / L$.
The second author \cite{kiefer} defined a family of meromorphic orthogonal modular forms of weight $\kappa>n$ generalising Oda's orthogonal variants \cite{oda} of Zagier's $f_{k,d}$'s \cite{zagierdoinaganuma} by
\[
    \omega_{\beta, m}^{\mero}(Z)=\sum_{\substack{\mu \in \beta + L \\ q(\mu) = m}} \frac{1}{\left(\mu,Z+e-q(Z)e'\right)^{\kappa}},\, Z\in\IH_n.
\]
These functions are meromorphic orthogonal modular forms with singularities of order $\kappa$ along certain algebraic cycles that vanish at the cusps. 

Let $\mathcal{C}\subset K\otimes \IR$ be the positive cone. An orthogonal modular form $f$ of weight $\kappa$ with respect to $\Gamma$ has a Fourier expansion of the following form 
\[
 f(Z) = \sum_{\lambda\in K' \cap \overline{\mathcal{C}}} a(\lambda) e^{2\pi i (\lambda,Z)}.
\]
Let $N$ be the level of $L$, let $\ell\in\IZ$ and let $\lambda_0\in K'\cap \bar{\mathcal{C}}$ be primitive.
We call a meromorphic orthogonal modular form $f$ \textit{magnetic} if the Fourier coefficients of index $\ell \lambda_0$ of $f$ are divisible by $(\ell N q(\lambda_0))^{\kappa-1}$.

Our main theorem is that the meromorphic orthogonal modular forms $\omega_{\beta, m}^{\mero}(Z)$ are magnetic.
\begin{Theorem}\label{thm:mainintro}
    Assume that the space of cusp forms of weight $k$ for the Weil representation $\rho_L$ is trivial. Then the functions $\omega_{\beta, m}^{\mero}(Z)$ are magnetic. 
\end{Theorem}

\begin{Remark}
    If the space of cusp forms of weight $k$ for the Weil representation $\rho_L$ is not trivial, we have to take suitable linear combinations of the functions $\omega_{\beta, m}^{\mero}(Z)$.
\end{Remark}

\begin{Remark}
    The case of signature $(2,2)$, i.e.\@ Hilbert modular forms, has been considered by Depouilly in \cite{depouilly2024divisibilitypropertiesfouriercoefficients}.
\end{Remark}

\begin{Remark}
Corollary \ref{corollary:divisibility} gives a general method to construct magnetic orthogonal modular forms by identifying vector-valued weakly holomorphic modular forms whose Fourier coefficients satisfy a given divisibility property, even if $n$ is odd, which corresponds to the weight $k$ being half-integral. 
Such weakly holomorphic modular forms were constructed by Li and Neururer in \cite[Theorem 1.4]{lineururer}. 
They used them as input functions of the additive Borcherds lift to produce magnetic modular forms for lattices of signature $(2, 1)$. 
These weakly holomorphic modular forms can also be used to construct magnetic modular forms for even lattices of general signature:
Let $L$ be an even lattice of signature $(2, n)$ such that the discriminant group $L' / L$ splits as an orthogonal sum $L' / L \simeq D_1 \oplus D_2$. 
Let $f$ be a weakly holomorphic modular form of weight $k$ with respect to the Weil representation $\rho_{D_1}$, that satisfies the divisibility property of Corollary \ref{corollary:divisibility}. 
Moreover, let $v \in \IZ[D_2]$ be an invariant vector with respect to the Weil representation $\rho_{D_2}$. 
For the construction of such invariant vectors, see \cite{Bieker,BiekerKiefer,MuellerInvariants,ZemelInvariants}. 
Then $f \otimes v$ is a weakly holomorphic modular form of weight $k$ with respect to $\rho_L \simeq \rho_{D_1} \otimes \rho_{D_2}$ satisfying the divisibility property of Corollary \ref{corollary:divisibility}. 
In particular, its additive Borcherds lift 
is magnetic.
\end{Remark}

Magnetic modular forms have deep connections to physics that extend far beyond the origin of their name.
They appear to emerge in quantum field theory when analyzing differential equations satisfied by dimensionally regulated Feynman integrals associated with specific one-parameter families of K3 surfaces. This phenomenon was first identified in \cite{poegelwangweinzierl} and very recently in \cite{duhrmaggionegasaueretal}. In \cite{duhrmaggionegasaueretal}, the authors further propose that orthogonal magnetic modular forms originate from geometries of Calabi-Yau type. Even though the authors have not proven the magneticity of the modular forms that appear, they have checked that several hundred of the Fourier coefficients have the desired property. In \cite{magneticconjectures} B\"onisch, Duhr, and Maggio further study (elliptic) modular forms and propose several conjectures for magnetic modular forms. We remark that the magnetic orthogonal modular forms that we consider in this work satisfy the conjecture of Bönisch, Duhr, and Maggio in \cite[Conjecture 3.1]{magneticconjectures} pertaining to the magneticity at every cusp; this is inherent by Borcherds' Theorem 14.3.

This paper is organized as follows. 
We briefly review the theory of modular forms for the Weil representation in Section \ref{sec:modforms}. In Section \ref{sec:orthforms} we introduce orthogonal modular forms and the functions $\omega_{\beta, m}^{\mero}$. In Section \ref{sec:thetaliftings} we recall Borcherds' Theorem 14.3 and compute the lifting of Poincar\'{e} series. We prove Theorem \ref{thm:mainintro} in Section \ref{sec:proof}.
%
\section*{Acknowledgements}
We thank the referee for valuable comments that substantially improved the article. Moreover, we thank Mads Christensen for pointing out to us the mod $p$ version of Sturm's theorem which led to an integral McGraw statement.

\section{Modular forms for the Weil representation}\label{sec:modforms}

In this section, we review the Weil representation and the notion of weakly holomorphic modular forms. 

Throughout, let \( L \) denote an even lattice equipped with a quadratic form \( q \) of signature \( (b^+, b^-) \). The associated bilinear form is represented by \( (\cdot, \cdot) \), and \( L' \) denotes the dual lattice.

\subsection{The Weil representation}  
We denote by \( \Mp_2(\IR) \) the metaplectic group, which is the double cover of \( \SL_2(\IR) \), given by pairs \( (M, \phi) \), where  
$
M = \left(\begin{smallmatrix}a & b \\ c & d \end{smallmatrix} \right)\in \SL_2(\IR)
$
and \( \phi : \IH \to \IC \) is a holomorphic square root of \( c \tau + d \). The subgroup \( \Mp_2(\IZ) \) is generated by the two elements  
\[
S = \left(\begin{pmatrix}0 & -1 \\ 1 & 0\end{pmatrix}, \sqrt{\tau}\right), \quad  
T = \left(\begin{pmatrix}1 & 1 \\ 0 & 1\end{pmatrix}, 1\right).
\]

Now, let \( L \) be an even lattice of signature \( (b^+, b^-) \), and consider the group ring \( \IC[L' / L] \) with basis \( (\frake_\beta)_{\beta \in L' / L} \) and inner product  
\[
\langle \frake_\beta, \frake_{\gamma} \rangle = \delta_{\beta, \gamma}.
\]  
This inner product is $\IC$-linear in the first argument and $\IC$-conjugate-linear in the second argument. Additionally, we use the notation  
\[
\frake_\beta(x) = e(x) \frake_\beta = e^{2 \pi i x} \frake_\beta.
\]

The Weil representation, denoted \( \rho_L : \Mp_2(\IZ) \to \GL(\IC[L' / L]) \), is a unitary representation defined by  
\begin{align*}
    \rho_L(T) \frake_\beta &= \frake_\beta(q(\beta)),\\
\rho_L(S) \frake_\beta &= \frac{\sqrt{i}^{b^- - b^+}}{\sqrt{|L' / L|}} \sum_{\gamma \in L' / L} \frake_{\gamma}(-(\beta, \gamma)).  
\end{align*}

We define \( L^- = (L, -q) \). The associated Weil representation \( \rho_{L^-} \) corresponds to the dual Weil representation $\rho_L^*$ of \( L \), since \(\rho_L^* =\overline{\rho_L}\).  

For $\beta,\gamma\in L'/L$ and $(M,\phi)\in\Mp_2(\Z)$ we denote the coefficient $\rho_{ \gamma\beta}(M,\phi)$ of the representation $\rho_L$ by
\[
\rho_{ \gamma\beta}(M) = \langle \rho_L(M,\phi) \frake_\beta, \frake_\gamma \rangle.
\]

\subsection{Modular forms for the Weil representation}

Let $k \in \frac{1}{2}\Z$. For functions $f : \IH \to \IC[L' / L]$ and $(M, \phi) \in \Mp_2(\IZ)$ we define the slash operator by
\[
(f \vert_{k, L} (M, \phi))(\tau) = \phi(\tau)^{-2k} \rho_L^{-1}(M, \phi) f(M \tau).
\]
Let $f : \IH \to \IC[L' / L]$ be a holomorphic function that satisfies $f \vert_{k, L} (M, \phi) = f$ for all $(M, \phi) \in \Mp_2(\IZ)$. Then $f$ has a Fourier expansion
$$f(\tau) = \sum_{\gamma \in L' / L} \sum_{\ell \in \IZ + q(\lambda)} c(\gamma, \ell) \frake_\gamma(\ell \tau).$$
We call such a function \emph{weakly holomorphic modular form of weight $k$ with respect to the Weil representation $\rho_L$}, if $c(\gamma, \ell) = 0$ for all but finitely many $\gamma \in L' / L, \ell < 0$ and we denote the space of weakly holomorphic modular forms by $M_{k, L}^!$. If $c(\gamma, \ell) = 0$ for all $\gamma \in L' / L, \ell < 0$, then $f$ is called a \emph{modular form} and we denote the space of modular forms by $M_{k, L}$. If additionally $c(\gamma, 0) = 0$ for all $\gamma \in L' / L$, then $f$ is called a \emph{cusp form} and the corresponding space is denoted by $S_{k, L}$.

If $f(\tau) = \sum_{\gamma \in L' / L} \sum_{\ell \in \IZ + q(\gamma)} c(\gamma, \ell) \frake_\gamma(\ell \tau)$ is a weakly holomorphic modular form, we call
$$\sum_{\gamma \in L' / L} \sum_{\substack{\ell \in \IZ + q(\gamma) \\ \ell < 0}} c(\gamma, \ell) \frake_\gamma(\ell \tau)$$
its \emph{principal part}.

\subsection{An integrality condition for weakly holomorphic modular forms}

In this section we prove an integral refinement of the rationality statement
deduced from McGraw's theorem. 

We recall the following well known statement that is a corollary of the work of McGraw \cite{mcgraw}.

\begin{Proposition}\label{prop:mcgraw}
    Let $k < 0, k \in \frac{1}{2} \IZ$. The space $M_{k, L}^!$ has a basis of weakly holomorphic modular forms with integral Fourier coefficients. In addition, if a weakly holomorphic modular form has integral principal part, then all its Fourier coefficients are rational with bounded denominator.
\end{Proposition}

We shall use the following form of Sturm's
theorem.

\begin{Theorem}{\cite[Theorem 1]{sturm}}\label{thm:sturm}
Let \(
g(\tau)=\sum_{n\geq 0} a(n)q^n
\)
be a holomorphic modular form of weight \(\kappa\) with respect to
\(\Gamma(N)\), and assume that all Fourier coefficients \(a(n)\) are integral.
Let \(p\) be a prime. If
\[
p\mid a(n)
\qquad
\text{for all } n\leq \frac{\kappa[\Gamma(1):\Gamma(N)]}{12},
\]
then
\[
p\mid a(n)
\qquad
\text{for all } n\geq 0.
\]
\end{Theorem}

As a consequence, we obtain the following integral version of McGraw's theorem.

\begin{Theorem}\label{thm:integral-mcgraw}
Assume that the Weil representation \(\rho_L\) descends to
\(SL_2(\mathbb Z)\), and choose \(N\) such that
\(
\Gamma(N)\subseteq \ker(\rho_L).
\)
Let
\(
f(\tau)=
\sum_{\gamma\in L'/L}
\sum_{\ell\in \mathbb Z+q(\gamma)}
c_f(\gamma,\ell)e(\ell\tau)\mathfrak e_\gamma
\in M_{k,L}^!
\)
be a weakly holomorphic modular form of weight \(k<0\). Let
\[
m:=
\min\left\{r\in\mathbb Z_{\geq 0} :
\Delta^r f\in M_{k+12r,L}\right\}.
\]
Put
\(\kappa=k+12m\) and
assume that
\(
c_f(\gamma,\ell)\in\mathbb Z
\)
for all \(\gamma\in L'/L\) and all
\(
\ell\leq
\max(-1,\frac{\kappa[\Gamma(1):\Gamma(N)]}{12}-m).
\)
Then all Fourier coefficients of \(f\) are integral.
\end{Theorem}

\begin{Remark}

The assumption that \(\rho_L\) descends to \(SL_2(\mathbb Z)\) is satisfied,
for example, when the signature of \(L\) is even. Under this hypothesis one
can choose \(N\) such that \(\Gamma(N)\subseteq\ker(\rho_L)\), and the
componentwise application of Sturm's theorem is legitimate. In the general
metaplectic case, the same proof works after replacing \(\Gamma(N)\) by a
finite-index subgroup of \(\operatorname{Mp}_2(\mathbb Z)\) on which
\(\rho_L\) is trivial and using the corresponding Sturm bound.

\end{Remark}
\begin{proof}
We let
\[
F:=\Delta^m f
=
\sum_{\gamma\in L'/L}
\sum_{\ell\in\mathbb Z+q(\gamma)}
b(\gamma,\ell)e(\ell\tau)\mathfrak e_\gamma .
\]
Then \(F\in M_{\kappa,L}\), where \(\kappa=k+12m\). By McGraw's theorem, $f$ has rational coefficients with bounded denominator.

Since $\Delta$ has integral Fourier coefficients,
the coefficient \(b(\gamma,\ell)\) is an integral linear combination of
coefficients \(c_f(\gamma,\ell-r)\) with \(r\geq m\) and therefore, also the coefficients $b(\gamma, \ell)$ are rational with bounded denominator. Moreover, if
\(
\ell\leq \frac{\kappa[\Gamma(1):\Gamma(N)]}{12},
\)
then
\(
\ell-r
\leq
\frac{\kappa[\Gamma(1):\Gamma(N)]}{12}-m
\)
and it follows that
\(
b(\gamma,\ell)\in\mathbb Z
\)
for all \(\gamma\in L'/L\) and all
\(\ell\leq \frac{\kappa[\Gamma(1):\Gamma(N)]}{12}.
\)

Since \(\Gamma(N)\subseteq\ker(\rho_L)\), every
component of a form in \(M_{\kappa,L}\) is a scalar modular form of weight
\(\kappa\) for \(\Gamma(N)\).

Let \(M\geq 1\) be the smallest positive integer such that \(MF\) has integral
Fourier coefficients. We show that \(M=1\). Suppose that \(M>1\), and let
\(p\mid M\) be a prime. Since the coefficients \(b(\gamma,\ell)\) are integral
for
\(
\ell\leq \frac{\kappa[\Gamma(1):\Gamma(N)]}{12},
\)
the corresponding coefficients of \(MF\) are divisible by \(p\). Applying
Sturm's theorem componentwise to \(MF\), we obtain that all Fourier
coefficients of \(MF\) are divisible by \(p\). Hence \((M/p)F\) still has
integral Fourier coefficients, contradicting the minimality of \(M\). Thus
\(M=1\), and \(F\) has integral Fourier coefficients.

The claim now follows by multiplying with $\Delta^{-1}$, which also has integral Fourier coefficients.
\end{proof}

\subsection{Differential operators}

Let $\tau\in\IH$, $\tau=u+iv$. Recall the weight raising and lowering operators 
\[
    R_k= 2i \frac{\partial}{\partial\tau} +kv^{-1},\quad L_k= -2iv^2\frac{\partial}{\partial\bar{\tau}}.
\]
They raise respectively lower the weight of a modular form by $2$.

The weight $k$ Laplace operator is given by
\[
\Delta_k=-R_{k-2}L_k.
\]
We also define an iterated version of the raising operator by $R_k^n=R_{k+2(n-1)}\circ\ldots\circ R_{k+2}\circ R_k$.

A \textit{harmonic weak Maass form of weight $k$ with respect to the Weil representation $\rho_L$} is a real-analytic function $f:\IH\to\IC $ that transforms like a modular form of weight $k$ for $\rho_L$, is annihilated by the Laplace operator $\Delta_k$, and grows at most linear exponentially as $v\to\infty$. We denote this space by $H_{k,L}$.

The differential operator
\[
\xi_k f = v^{k-2} \overline{L_k f} = R_{-k} v^k \overline{f}
\]
maps harmonic weak Maass forms of weight $k$ for the Weil representation $\rho_L$ to cusp forms of weight $2-k$ for the dual Weil representation $\rho^*_L$.

We further define $D:=\frac{1}{2\pi i }\frac{\partial}{\partial \tau}$. Let $k\geq 2$ be an integer. The Bol operator
\[
 D^{k-1}= \frac{1}{(-4\pi)^{k-1}} R_{2-k}^{k-1} 
\]
maps harmonic weak Maass forms of weight $2-k$ to weakly holomorphic modular forms of weight $k$ with vanishing constant term at all cusps. Moreover, the image consists of those weakly holomorphic modular forms that are orthogonal to cusp forms with respect to the regularised Petersson inner product (see \cite[Theorem 1.2]{bruinieronorhoades}).

\begin{Lemma}{\cite[Theorem 1.1]{bruinieronorhoades}}
    Let $f\in H_{2-k,L}$ be a harmonic weak Maass form of weight $2-k, \, k \geq 2$, with holomorphic part $f^+(\tau) = \sum_{\gamma \in L' / L} \sum_{\ell \in \IZ + q(\gamma)} c(\gamma, \ell) \frake_\gamma(\ell \tau)$. Then the Fourier expansion of $D^{k-1} f$ is given by 
    \[
     D^{k-1} f = D^{k-1}f^+ = \sum_{\gamma \in L' / L} \sum_{\ell \in \IZ + q(\gamma)} c(\gamma, \ell) \ell^{k-1} \frake_\gamma(\ell \tau).
    \]
\end{Lemma}

\begin{Corollary}\label{corollary:divisibility}
    Let $k\geq 2$ and let $f \in H_{2-k,L}$ be a harmonic weak Maass form whose Fourier coefficients of the holomorphic part $f^+(\tau) = \sum_{\gamma \in L' / L} \sum_{\ell \in \IZ + q(\gamma)} c(\gamma, \ell) \frake_\gamma(\ell \tau)$ are integral. Then the Fourier coefficients $b(\gamma, \ell)$ of $N^{k-1} D^{k-1} f$ satisfy $(N\ell)^{k-1} \mid b(\gamma, \ell)$, where $N$ denotes the level of the lattice $L$.
\end{Corollary}

\begin{Remark}
    In general, it is not expected that harmonic weak Maass forms have integral Fourier coefficients. Apart from Ramanujan's theta functions, whose Fourier coefficients admit a combinatorial interpretation, there are very few examples of harmonic weak Maass forms known whose coefficients are not generically transcendent (compare \cite{bruinieronoannals} and \cite{ehlenlischwagenscheidt}).
    Using the result of \cite{mcgraw}, there are many weakly holomorphic modular forms with this property.
\end{Remark}

\subsection{Poincar\'e series}

Let $k\in\frac{1}{2}\IZ$ and let $\kappa = \frac{b^-}{2} - \frac{b^+}{2} + k$. By $W_{\nu, \mu}(v), M_{\nu,\mu}(v)$ we denote the usual Whittaker functions (see \cite[Chapter 13.4]{nist}).

For $s\in\IC$ and $v \in\IR^\times$ we let
\begin{align*}
    \mathcal{M}_{k, s}(v) = \lvert v \rvert^{-k/2} M_{\sgn(v)k/2,s-1/2}(\lvert v \rvert), \\
    \calW_{k, s}(v) = \lvert v \rvert^{-k/2} W_{\sgn(v) k/2,s-1/2}(\lvert v \rvert).
\end{align*}
These functions are eigenfunctions of $\Delta_k$ of eigenvalue $s(1-s)+(k^2-2k)/4$.
The harmonic points are given by $s = k/2$ and $s = 1 - k/2$. We have $\calM_{k, k/2}(v) = e^{-v/2}$ and for $v \in \IR_{>0}$ we have $\calW_{k, k/2}(v) = \calW_{2-k, 1 - (2 -k)/2}(v) = e^{-v/2}$ (compare \cite[Equation(2.20)]{jeonkangkim} and note our different normalisation of $\calW_{k,s}$).

For $k\in\frac12\IZ$, $2k + b^- - b^+ = 2 \kappa \equiv 0 \pmod{2}$, $\beta\in L'/L$, and $m\in \IZ+q(\beta),\, m \neq 0$ we define the \emph{Poincar\'e series} $\mathcal{F}_{\beta, m, k}(\tau,s)$ for the Weil representation $\rho_{L}$ by
\[
    \mathcal{F}_{\beta, m, k}(\tau,s) =\frac{1}{2\G(2s)} \sum_{(M, \phi) \in \tilde{\Gamma}_\infty \backslash \Mp_2(\Z)} \left(\mathcal{M}_{k,s}(4\pi m v)e(mu) \frake_\beta\right) \vert_{k, L}(\gamma, \phi).
\]
It converges for $\Re(s)>1$ and is annihilated by the weight $k$ Laplace operator $\Delta_k$ at the harmonic points $s=1-\frac{k}{2}, \frac{k}{2}$. 
We will be interested in the special values
\[
\calF_{\beta, m, k}(\tau) = \calF_{\beta, m, k}(\tau, k/2), \qquad \calF_{\beta, m, 2-k}(\tau) = \calF_{\beta, m, 2-k}(\tau, k/2).
\]
In the first two cases, we obtain cusp forms for $m > 0$ and weakly holomorphic modular forms for $m < 0$. In the other two cases, we have harmonic weak Maass forms for $m < 0$.

Under the Bol operator, Poincar\'{e} series map to Poincar\'{e} series. 
\begin{Lemma}[Theorem 6.11 iii) in \cite{onobook}]\label{lemma:bol}
    For $k\in2 \N$ we have
    \[
    D^{k-1} \calF_{\beta,m,2-k} (\tau)=m^{k-1} \calF_{\beta,m,k}(\tau). 
    \]
\end{Lemma}

\section{Orthogonal modular forms}\label{sec:orthforms}
In this section we define the orthogonal upper half plane, orthogonal modular forms, and  the functions $\omega$. We follow the exposition in \cite{kiefer}.

\subsection{The orthogonal upper half plane}
We let $L$ be an even lattice of signature $(2,n)$, $n>0$, with quadratic form $q$. By $(\cdot,\cdot)$ we denote the associated bilinear form. 
Let $u\in L$ be a primitive isotropic vector, and let $u'\in L$ with $(u,u')=1$. 
Moreover, we let $\tilde{u}'=u'-q(u')u$. 
We define a lattice $K$ by
\[
K=L\cap u^\perp \cap u'^\perp,
\]
and write $V=L\otimes \IQ$, $V(\IR)=L\otimes \IR$, $V(\IC)=L\otimes\IC$, and $W=K\otimes\IQ$, $W(\IR)=K\otimes\R$, $W(\IC)=K\otimes\IC$, respectively. 
We let
\[
    \IH_n^\pm =\left\{ Z=X+iY\in W(\IC)\,:\, q(Y)>0\right\},
\]
which has two connected components. We denote one of them by $\IH_n$. 
It is called \emph{orthogonal upper half plane}. 
Then $\IH_n = K \otimes \IR \oplus i \calC$, where $\calC \subseteq K \otimes \IR$ is a connected component of the \emph{positive cone}.

We let $\mathbb{P}(V(\IC))$ be the associated projective space and denote the canonical projection by $V(\IC)\to \mathbb{P}(V(\IC)),\, Z\to\left[ Z\right]$. 
The map 
\[
    \psi:\IH^\pm_n \to V(\IC),\, Z \mapsto \psi(Z)=Z-q(Z)u+\tilde{u}'
\]
induces a biholomorphic map 
\begin{equation}\label{eq:psi}
    \IH_n^\pm \to \left\{ \left[Z\right]=\left[X+iY\right] \in\mathbb{P}(V(\IC))\,:\, (Z,Z)=0,\, (Z,\overline{Z})>0\right\}. 
\end{equation}
The orthogonal group $O(V(\IR))$ has a natural action on the right side, giving an action on $\IH_n^\pm$ by $\sigma Z:=\psi^{-1} (\sigma\psi(Z))$. 

We define the factor of automorphy $j(\sigma,Z)$ by
\[
j(\sigma, Z)=(u,\sigma(\psi(Z))), \text{ for } \sigma \in O(V(\IR)), \text{ and }Z\in\IH_n^\pm. 
\]
Then $j(\sigma,Z) \psi(\sigma Z)=\sigma \psi(Z)$.

Let
\[
\mathrm{Gr}(L)=\mathrm{Gr}(V)=\left\{ (Z^+,Z^-)\,:\, Z^+,\, Z^-\subset V, \, Z^+\oplus Z^-=V,\, q|_{Z^+}>0,\, q|_{Z^-}<0\right\}
\]
be the Grassmannian associated to $\IH_n$ that is isomorphic to the orthogonal upper half-plane via
\[     \IH_n\to \mathrm{Gr}(L),\, Z\mapsto (Z^+, Z^-) \coloneqq (\langle \Re(\psi(Z)), \Im(\psi(Z)) \rangle,\langle \Re(\psi(Z)), \Im(\psi(Z)) \rangle^\perp). \]
Let $Z\in\IH_n$. We write $\lambda_{Z^\pm}$ for the projection of $\lambda\in V(\R)$ onto $Z^\pm\subset V(\IR)$.

\subsection{Algebraic cycles}
Let $\nu\in V$ with $q(\nu)<0$. We consider the complex submanifold 
\[
T_\nu=\left\{ Z\in \IH_n\,:\, q(\nu_{Z^+})=0\right\}.
\]
It corresponds to the set of $(Z^+,Z^-)\in\mathrm{Gr}(L)$ with $\nu\in Z^-$ and thus to the Grassmannian of the real quadratic space $\nu^\perp$ of signature $(2,n-1)$ which is isomorphic to $\IH_{n-1}$ and has complex codimension $1$ in $\IH_n$.

The hypersurfaces \(T_\nu\) are the rational quadratic divisors in the
orthogonal upper half-plane. After passing to an arithmetic quotient and
summing over the appropriate lattice cosets, they define algebraic divisors,
usually called Heegner divisors or special divisors; see
\cite[Chapter~5]{brhabil}. In the case of signature \((2,2)\), these rational quadratic divisors are
the classical Hirzebruch--Zagier divisors on Hilbert modular surfaces;
see \cite{hirzebruchzagier,hlr}.

\subsection{Orthogonal modular forms}
By $G=O^+(V(\IR))$ we denote the subgroup of the group $O(V(\IR))$ that acts on the orthogonal upper half-plane $\IH_n$. 
It further acts on a function $h : \IH_n\to\IC$ by the \emph{weight $\kappa\in \IZ$ slash operator} 
\[
    \left(h|_\kappa \sigma\right)(Z)=j(\sigma,Z)^{-\kappa} h(\sigma Z),\, \sigma\in O^+(V(\IR)).
\]
Let $\G\subset O^+(L)$ be the discriminant kernel, i.e. the subgroup that acts trivially on the cosets $L' / L$.
A holomorphic function $h : \IH_n\to\IC$ is called an \emph{orthogonal modular form of weight $\kappa\in\IZ$ with respect to $\G$} if $h|_\kappa \sigma=h$ for all $\sigma\in\G$.
If $n$ is bigger than the Witt rank, orthogonal modular forms are holomorphic at the cusps by the Koecher principle, see \cite[Proposition 4.15]{brhabil}. This is, in particular, satisfied for $n \geq 3$.

\subsection{The functions $\omega$}
Following Kiefer \cite{kiefer}, we define meromorphic orthogonal modular forms for $\nu \in V$ with $q(\nu)<0$ and $\kappa>n$ by
\[
 \omega_{\nu,\kappa}^{\mathrm{mero}} (Z)=\sum_{\gamma \in \G_{\nu}\setminus\G} \frac{1}{(\nu,\psi(Z))^\kappa}\bigg|_\kappa \gamma,
\]
where $Z\in \IH_n$ and $\psi(Z)$ was defined in \eqref{eq:psi}. Here, $\Gamma_\nu$ is the stabilizer of $\nu$ in $\Gamma$. These are orthogonal meromorphic modular forms with singularities along the $\G$-translates of the algebraic cycle $T_\nu$ of order $\kappa$ that vanish at the cusps. 

Their well-definedness follows from
\[
    (\mu, \psi(\gamma Z))=\frac{(\gamma^{-1}\mu,\psi(Z))}{j(\gamma,Z)}=\frac{(\mu,\psi(Z))}{j(\gamma,Z)} \text{ for }\gamma\in\G_\mu.
\]
In \cite{oda} Oda defined cuspidal versions of the functions $ \omega_{\nu,\kappa}^{\mathrm{mero}} (Z)$ which are itself generalisations of Zagier's Hilbert modular forms $\omega_m$, respectively, classical cusp forms $f_{k,D}$. To simplify notation, we usually drop the index $\kappa$, if the weight is clear from the context. For $\beta \in L' / L$ and $m < 0$ we set
$$\omega_{\beta, m}^{\mero}(Z) = \sum_{\nu \in \Gamma \bs L_{\beta, m}} \omega_{\nu}^{\mero}(Z),$$
where $L_{\beta, m} = \{ \lambda \in \beta + L \ \vert\ q(\lambda) = m\}$.

\section{Theta liftings}\label{sec:thetaliftings}
Let $L$ be an even lattice of signature $(2,n)$, $n\in\IN$. Let $\kappa\in\IZ$ be an integer satisfying $\kappa>n$ and let $k=1-n/2+\kappa$.

\subsection{A theta function of Borcherds}
We define the following theta function
\begin{align}\label{def:odatheta}
    \Theta_L(\tau, Z) = v^{n/2} \sum_{\lambda \in L'} \left(\frac{(\lambda, \psi(Z))}{q(Y)}\right)^\kappa \frake_\lambda(q(\lambda_{Z^+}) \tau + q(\lambda_{Z^-}) \overline{\tau}), \, \tau\in\IH,\, Z\in\IH_n.
\end{align}
It follows from \cite[Theorem 4.1]{borcherds} that $\Theta_L(\tau, Z)$ transforms as a modular form of weight $k$ for the Weil representation $\rho_L$ in the variable $\tau$ and its complex conjugate transforms as an orthogonal modular form of weight $\kappa$ in the variable $Z$.

\subsection{The Fourier expansion of the additive Borcherds lift}

Let $k \in\frac12\IZ$ and $k>2$. We define the regularised additive Borcherds lift of a weakly holomorphic modular form $f \in M^{\text{!}}_{k,L}$ by
\begin{align*}
\Phi_L(f,Z) &= \frac{i^\kappa}{2^{\kappa+1}}\int^{\text{reg}}_{\mathcal{F}}\left\langle f(\tau),\Theta_L(\tau,Z)\right\rangle v^{k} d\mu(\tau)
\\
&= \frac{i^\kappa}{2^{\kappa+1}}\lim_{T\to\infty} \int_{\mathcal{F}_T}\left\langle f(\tau),\Theta_L(\tau,Z)\right\rangle v^{k} d\mu(\tau) ,
\end{align*}
where $\mathcal{F}$ is the standard fundamental domain for $\SL_2(\Z)\backslash\H$, $\mathcal{F}_T$ denotes the truncated fundamental domain, and $d\mu(\tau) = \frac{du \, dv}{v^2}$ is the invariant measure on $\H$.

This lift has also been considered by Oda in \cite{oda}.

We recall the statement of \cite[Theorem 14.3]{borcherds} which gives the Fourier expansion of the theta lift $\Phi_L(f,Z)$.

\begin{Theorem}{\cite[Theorem 14.3]{borcherds}}\label{theorem:borcherds}
    Let $f(\tau) = \sum_{\gamma \in L' / L} \sum_{\ell \in \IZ + q(\gamma)} c(\gamma, \ell) \frake_\gamma(\ell \tau)\in M^{\text{!}}_{k,L}$ be a weakly holomorphic modular form of weight $k$ with respect to the Weil representation $\rho_L$ and assume that $\kappa = \frac{n}{2} - 1 + k > 1$. 
    
    Let $N_{u} \in \IN$ be the level of $u$ and choose $\zeta \in L$ with $(u, \zeta) = N_{u}$. Then the Fourier expansion of $\Phi_L(f, Z)$ for sufficiently large $|Y|$, where $Z=X+iY\in K\otimes \R +i\calC$, is given by
    \begin{align*}
        &- \sum_{m, m' = 1}^{N_{u}} N_{u}^{\kappa - 1} c\left(\frac{m u}{N_{u}}, 0\right) e\left(\frac{mm'}{N_{u}}\right) \frac{B_\kappa(m' / N_{u})}{2 \kappa} \\
        &+ \sum_{\substack{\lambda \in K' \\ \lambda \in \overline{\calC}}} \sum_{m' = 1}^{N_{u}}  \sum_{m \mid \lambda} m^{\kappa - 1} c\left(\frac{\lambda}{m} - \frac{(\lambda, \zeta)u}{m N_{u}} + \frac{m' u}{N_{u}}, \frac{q(\lambda)}{m^2}\right) e\left(\frac{mm' - (\lambda, \zeta)}{N_{u}}\right) e(\lambda, Z),
    \end{align*}
    where $\overline{\calC}$ is the closure of the positive cone $\calC$.
\end{Theorem}

The theorem implies divisibility criteria for the Fourier coefficients of the theta lift $\Phi_L(f,Z)$ of weakly holomorphic modular forms $f$.

\begin{Corollary}\label{corollary:divborcherds}
    Let $f(\tau) = \sum_{\gamma \in L' / L} \sum_{\ell \in \IZ + q(\gamma)} c(\gamma, \ell) \frake_\gamma(\ell \tau)$ be a weakly holomorphic modular form of weight $k$ with respect to the Weil representation $\rho_L$. Let $\kappa = \frac{n}{2} - 1 + k > 1$.
    
    \begin{enumerate}
        \item Under the assumption $c(\gamma, \ell) \in \IZ$ for all $\gamma \in L' / L, \ell \in \IZ + q(\gamma)$, the Fourier coefficients of $\Phi_L(f, \cdot)$ are in $\IZ[e(1/N_{u})]$, i.e.\@ they are cyclotomic integers.
        \item  Suppose in addition that $c(\gamma, \ell)$ is divisible by $(N\ell)^{r-1}$ for all $\gamma \in L' / L, \ell \in \IZ + q(\gamma)$, where $N$ is the level of $L$ and $\frac{n}{2} + k - r - 1 \geq 0, r \in \IZ$. Let $\lambda_0 \in K' \cap \overline{\calC}$ be primitive. Then the Fourier coefficient of $\Phi_L(f, \cdot)$ of index $\lambda = \ell \lambda_0$ is divisible by $(\ell N q(\lambda_0))^{r - 1}$.
    \end{enumerate}
\end{Corollary}

\begin{proof}
    The claim in (1) follows directly from the Fourier expansion of $\Phi_L(f, \cdot)$. For (2) we use that $\kappa = \frac{n}{2} - 1 + k$, so that the divisor sum is given by
    \begin{align}
        \sum_{m \mid \lambda} m^{\frac{n}{2} + k - 2} c\left(\frac{\lambda}{m} - \frac{(\lambda, \zeta)u}{m N_{u}} + \frac{m' u}{N_{u}}, \frac{q(\lambda)}{m^2}\right). \label{eq:SummandFormula}
    \end{align}
    Now, the Fourier coefficients $c(*, \frac{q(\ell \lambda_0)}{m^2})$ are divisible by $\left(N q(\lambda_0) \frac{\ell^2}{m^2}\right)^{r-1}$. Hence, the divisor sum in \eqref{eq:SummandFormula} for $\lambda = \ell \lambda_0$ is given by
    $$\sum_{m \mid \ell} m^{\frac{n}{2} + k - 2} \left(N q(\lambda_0) \frac{\ell^2}{m^2}\right)^{r-1} c_m = \left(N q(\lambda_0) \ell \right)^{r-1} \sum_{m \mid \ell} m^{\frac{n}{2} + k - r - 1} \left(\frac{\ell}{m}\right)^{r - 1} c_m$$
    for some integers $c_m$ (depending on $\lambda, \ell$ and $m$). Since $m \mid \ell$ and $\frac{n}{2} + k - r - 1 \geq 0$ is integral, we obtain that every summand in the Fourier coefficient is divisible by $(N q(\lambda) \ell)^{r - 1}$.
\end{proof}

\subsection{Borcherds' additive regularised theta lift of weakly holomorphic Poincar\'{e} series}

We compute the lift of weakly holomorphic Poincar\'{e} series.

\begin{Theorem}\label{theorem:liftpc}
Let $k \in\frac12\IZ, k>2$ and $m<0$. Set $\kappa = \frac{n}{2} - 1 + k$
    The regularised additive Borcherds lift of the weakly holomorphic Poincar\'{e} series $\calF_{\beta,m,k}(\tau)$ is given by
    \[
     \Phi_L(\calF_{\beta,m,k},Z) = (-2\pi i)^{-\kappa}\frac{(\kappa-1)!}{2(k-1)!} 
     \omega_{\beta, m}^{\mero}(Z).
     \]
\end{Theorem}

\begin{proof}
   We compute the following integral
\[
\frac{i^\kappa}{2^{\kappa+1}}
\frac{1}{2\G(k)}\int_{\mathcal{F}}^{\reg} \!\!\!\!\!\sum_{(M,\phi)\in \tilde{\Gamma}_\infty\setminus \Mp_2(\Z)} \left\langle  \left[\calM_{k,k/2}(4\pi mv)e(mu)\mathfrak{e}_\beta\right]|_{k,L} (M,\phi),{\Theta_{L}(\tau,Z)}\right\rangle v^{k}d\mu(\tau)
\]
following the strategy of the proof of \cite[Theorem~2.14]{brhabil}. We first apply the unfolding trick and insert the Fourier expansion of the theta function to get
\begin{align*}
C_{k,n}
\frac{1}{q(Y)^\kappa}\int_{v=0}^\infty \int_{u=0}^1 \sum_{\lambda\in L+\beta} v^{\frac{n}{2}+k-2} \calM_{k,k/2}(4\pi m v)e(m u) \overline{\left(\lambda,\psi(Z)\right)^{\kappa}} e(-q(\lambda)u)e^{-2\pi v q_Z(\lambda)} dudv,
\end{align*}
where $C_{k,n}=\frac{i^\kappa}{2^{\kappa+1}\G(k)}$ and $q_Z(\lambda) = q(\lambda_{Z^+}) - q(\lambda_{Z^-})$. The integral over $u$ vanishes unless $m=q(\lambda)$, in which case it equals $1$. We plug in the definition of $\calM_{k,k/2}$ and note that the resulting integral
\begin{align*}
 \int_{v=0}^\infty v^{\frac{n+k}{2}-2} M_{-k/2,k/2-1/2}(4\pi |m|v)e^{-2\pi v q_Z(\lambda)}dv
\end{align*}
is a Laplace transform that equals
\begin{equation}\label{eq:liftpc}
    (4\pi |m|)^{\frac{k}{2}} \G(\kappa) \left(2\pi (q_Z(\lambda)+|m|)\right)^{-\kappa} 
    {}_2F_1\left(\kappa,k,k;\frac{2|m|}{q_Z(\lambda)+|m|}\right)
\end{equation}
by equation (11) on p.\@ 215 of \cite{laplacetransform}. Note that ${}_2F_1(a,b,b;z)=(1-z)^{-a}$ by equation (15.4.6) of \cite{nist}. Therefore, the expression in \eqref{eq:liftpc} simplifies to
\[
(4\pi |m|)^{\frac{k}{2}} \G(\kappa) \left(2\pi(q_Z(\lambda)+m)\right)^{-\kappa},
\]
where we used that $m<0$. 
 Noting that $q_Z(\lambda)+m=2q(\lambda_{Z^+})$ we obtain
\[
 \Phi_L(\calF_{\beta,m,k},Z) = \frac12(-2\pi i )^{-\kappa} 4^{-\kappa}  \frac{(\kappa-1)!}{(k-1)!}\frac{1}{q(Y)^\kappa}
 \sum_{\substack{\lambda \in L+\beta\\q(\lambda)=m}} \overline{(\lambda,\psi(Z))}^\kappa q(\lambda_{Z^+})^{-\kappa}.
 \]
We use
\[
q(\lambda_{Z^+}) = \frac{\lvert (\lambda, \psi(Z)) \rvert^2}{4 q(Y)}
\]
to obtain the statement in the theorem.
\end{proof}

\section{Proof of Theorem \ref{thm:mainintro}}\label{sec:proof}

By Proposition \ref{prop:mcgraw} the space of weakly holomorphic modular forms of weight $2-k$ for $\rho_L$ has a basis with integral Fourier coefficients.
These basis elements can be written as linear combinations of Poincar\'{e} series $\calF_{\beta,m,2-k}$ with $m<0$ as these Poincar\'{e} series provide a basis of the space of weakly holomorphic modular forms of weight $2-k$ for $\rho_L$. 

By Lemma \ref{lemma:bol} we have for 
\[
f=\sum_{\beta \in L'/L} \sum_{n<0} c_f(n,\beta) \calF_{\beta,m,2-k}(\tau)\in M^{!}_{2-k,L}
\]
that 
\[
    D^{k-1} f= \sum_{\beta \in L'/L} \sum_{n<0} c_f(n,\beta) n^{k-1} \calF_{\beta,m,k}(\tau).
\]
If $f$ has integral Fourier coefficients we see by Corollary \ref{corollary:divisibility} that the Fourier coefficients of $N^{k-1}D^{k-1} f$ of index $(\ell,\beta)$ are divisible by $(N\ell)^{k-1}$. 
By Theorem \ref{theorem:borcherds} and Corollary \ref{corollary:divborcherds} the Fourier coefficients of the lifts $\Phi_L(N^{k-1} D^{k-1} f,Z)$ satisfy the desired divisibility property.

By Theorem \ref{theorem:liftpc} the lifts of $D^{k-1} \calF_{\beta,m,2-k}$ are given by (multiples of) the functions $\omega_m^{\mero}$. 
This implies that the lift of a weakly holomorphic modular form $f\in M^{!}_{k,L}$ as above is given by the corresponding linear combination of the functions $\omega_{\beta, m}^{\mero}$. 

If the dual space of cusp forms $S_{k,L^-}$ is empty, then $H_{2-k,L}=M^{\text{!}}_{2-k,L}$ and we obtain the statement in Theorem \ref{thm:mainintro}.


\bibliographystyle{acm}
\bibliography{bib.bib}

\end{document}